\documentclass[12pt]{amsart}
\usepackage{amssymb}
\usepackage{mathtools}
\usepackage[english]{babel}
\usepackage{epsfig}
\usepackage{tikz}
\usepackage{mathptmx}
\usepackage{amsmath,amsfonts,amssymb}
\usepackage{mathtools}
\usepackage{mathrsfs}
\usepackage[all]{xy}
\usepackage{graphicx}
\usepackage{latexsym}
\usepackage{verbatim}

\numberwithin{equation}{section}

\def\Re{{\sf Re}\,}
\def\Im{{\sf Im}\,}

\newcommand{\D}{\mathbb D}

\newcommand{\R}{\mathbb R}
\newcommand{\Ha}{\mathbb H}

\newcommand{\C}{\mathbb C}

\def\Re{{\sf Re}\,}
\def\Im{{\sf Im}\,}

\def\Re{{\sf Re}\,}
\def\Im{{\sf Im}\,}

\def\Re{{\sf Re}\,}
\def\Im{{\sf Im}\,}

\def\1#1{\overline{#1}}
\def\2#1{\widetilde{#1}}
\def\3#1{\widehat{#1}}
\def\4#1{\mathbb{#1}}
\def\5#1{\frak{#1}}
\def\6#1{{\mathcal{#1}}}

\def\Re{{\sf Re}\,}
\def\Im{{\sf Im}\,}

\newcommand{\mcite}[1]{\csname b@#1\endcsname}

\theoremstyle{plain}

\def\ha{\omega}

\def\Re{{\sf Re}\,}
\def\Im{{\sf Im}\,}

\def\Arg{{\rm Arg}}

\newtheorem{theorem}{Theorem}[section]
\newtheorem{lemma}[theorem]{Lemma}
\newtheorem{proposition}[theorem]{Proposition}

\theoremstyle{definition}
\newtheorem{definition}[theorem]{Definition}

\theoremstyle{remark}
\newtheorem{remark}[theorem]{Remark}

\numberwithin{equation}{section}

\title[Asymptotic monotonicity ]{Asymptotic monotonicity of the orthogonal speed and rate of convergence for semigroups of holomorphic self-maps of the unit disc}

\author[F. Bracci]{Filippo Bracci$^\dag$}
\address{F. Bracci: Dipartimento di Matematica, Universit\`a di Roma ``Tor Vergata", Via della Ricerca
Scientifica 1, 00133, Roma, Italia.} \email{fbracci@mat.uniroma2.it}

\author[D. Cordella]{Davide Cordella$^\dag$}
 \address{D. Cordella: Dipartimento di Matematica, Universit\`a di Roma ``Tor Vergata", Via della Ricerca
Scientifica 1, 00133, Roma, Italia.} 
 \email{cordella@mat.uniroma2.it}

\author[M. Kourou]{Maria Kourou}
\address{M. Kourou: Department of Mathematics,
  University of W\"urzburg, Emil Fischer Strasse 40, 97074, W\"urzburg,
  Germany. {\sl Current Address:} Department of Mathematics,
  Aristotle University of Thessaloniki, 54124, Thessaloniki, Greece.} \email{mkouroue@math.auth.gr}


\thanks{$^\dag$Partially supported by PRIN {\sl Real and Complex
Manifolds: Topology, Geometry and holomorphic dynamics} n.2017JZ2SW5, by INdAM and  by the MIUR Excellence Department Project awarded to the
Department of Mathematics, University of Rome Tor Vergata, CUP E83C18000100006}

\long\def\ReM#1{\relax}

\begin{document}

\selectlanguage{english}
\begin{abstract}
We show that the orthogonal speed of semigroups of holomorphic self-maps of the unit disc is asymptotically monotone in most cases. Such a theorem allows to generalize previous results of D. Betsakos and D. Betsakos, M. D. Contreras and S. D\'iaz-Madrigal and to obtain new estimates for the rate of convergence of orbits of semigroups.
\end{abstract}
\maketitle

\tableofcontents


\section{Introduction}

The theory of continuous semigroups of holomorphic self-maps of the unit disc $\D:=\{z\in \C: |z|<1\}$---or just, for short, {\sl semigroups in $\D$}---is a flourishing subject of study since the early nineteen century, both as a subject by itself and for many different applications, see, {\sl e.g.},   \cite{Ababook89, A1,    BerPor78,  BCDbook, EliShobook10,  Shobook01,  Sis98} and bibliography therein.

In this paper we are interested in considering the so-called ``rate of convergence'' of the orbits of a non-elliptic semigroup in $\D$ to its Denoy-Wolff point. Estimates for the rate of convergence of an orbit of a non-elliptic semigroup in $\D$ have been obtained in \cite{Bet, Bet2, BetCD, EJ, EKRS, A5, ElSh, A7}. 

In particular, in \cite{Bet}, D. Betsakos proved that if $(\phi_t)$ is a semigroup in $\D$ with Denjoy-Wolff point $\tau\in \partial \D$, then there exists a constant $K>0$ such that
\begin{equation}\label{Betsakos}
|\phi_t(0)-\tau|\leq K t^{-1/2}, \quad t\geq 0.
\end{equation}
The point $0$ can be easily replaced with any $z\in \D$. However,  the exponent $-1/2$ of $t$ is sharp, and can be  replaced with $-1$ in case $(\phi_t)$ is either hyperbolic or parabolic with positive hyperbolic step. 

In \cite[Thm. 5.3]{BetCD} (see also \cite[Thm. 16.3.1]{BCDbook}), D. Betsakos, M. D. Contreras and S. D{\'i}az-Madrigal, got an estimate of the previous type with the exponent $-1/2$ replaced by  $-\frac{\pi}{\alpha+\beta}$  in case the image of the Koenigs function of the semigroup is contained in a sector of the form $W_{\alpha, \beta}:=p+i\{re^{i\theta}: r>0, -\alpha<\theta<\beta\}$ with $\alpha, \beta\in [0,\pi]$ and $\alpha+\beta>0$ for some---completely irrelevant for this discussion---point $p\in \C$.

In \cite{Br} (see also \cite[Ch. 16]{BCDbook}), the first named author introduced three quantities, called {\sl speeds}, which are defined in intrinsic terms using the hyperbolic distance and showed that the previous estimates can be translated in terms of one of such speeds. To be more concrete, let $\tau\in\partial \D$ be the Denjoy-Wolff point of $(\phi_t)$. Let $\pi(\phi_t(0))\in (-1,1)\tau$ be the closest point to $\phi_t(0)$ in the sense of hyperbolic distance $k_\D$ in $\D$. For $t\geq 0$, we let
\[
v^o(t):=k_\D(0, \pi(\phi_t(0))),
\]
and call it the {\sl orthogonal speed} of $(\phi_t)$. It can be shown that $v^o(t)\sim -\frac{1}{2}\log |\tau-\phi_t(0)|$, and therefore \eqref{Betsakos} can be translated in 
\begin{equation}\label{Eq:hypBet}
\liminf_{t\to +\infty}[v^o(t)-\frac{1}{4}\log t]>-\infty,
\end{equation}
and, similarly, the estimate in \cite[Thm. 5.3]{BetCD} can be obtained by replacing $\frac{1}{4}$ by $ \frac{\pi}{2(\alpha+\beta)}$. 

Now,  in \cite[Prop. 6.5]{Br} (see also \cite[Cor. 16.2.6]{BCDbook}) it is proved that the orthogonal speed of a semigroup whose image under the Koenigs function is a sector $W_{\alpha,\beta}$, goes like $-\frac{\pi}{2(\alpha+\beta)}\log t$ as $t\to +\infty$. Therefore, \eqref{Eq:hypBet} and \cite[Thm. 5.3]{BetCD} can be rephrased as $\liminf_{t\to +\infty}[v^o(t)-w^o(t)]>-\infty$, where $w^0(t)$ is the orthogonal speed of the semigroup whose image under the Koenigs function is a sector $W_{\alpha,\beta}$. Hence, the following natural question was raised in \cite{Br} (see {\sl Question~4} in \cite[Sec. 8]{Br}):

\medbreak
{\sl Question:} Let $(\phi_t)$ and $(\tilde\phi_t)$ be non-elliptic semigroups in $\D$ with Koenigs functions $h$ and $\tilde h$, respectively, and denote by $v^o(t)$ ({\sl resp.} $\tilde v^o(t)$) the orthogonal speed of $(\phi_t)$  ({\sl resp.} $(\tilde\phi_t)$). Assume $h(\D)\subset \tilde h(\D)$. Is it true that 
\[
\liminf_{t\to +\infty}[v^o(t)-\tilde v^o(t)]>-\infty?
\]
In other words, is the orthogonal speed asymptotically monotone?
\medbreak

In this paper we give a (partial) affirmative answer to the previous question. In particular, we prove that if one replaces the $\liminf$ with $\limsup$, the answer is always yes.

\begin{theorem}\label{Thm:almost-asymp-mono}
Let $(\phi_t), (\tilde \phi_t)$ be non-elliptic semigroups in $\D$. Let $h$ ({\sl respectively}, $\tilde h$) be the Koenigs function of $(\phi_t)$ ({\sl resp.} of $(\tilde \phi_t)$). Suppose that $h(\D)\subset \tilde h(\D)$. Then 
\[
\limsup_{t\to +\infty}[v^o(t)-\tilde v^o(t)]>-\infty,
\]
or, equivalently,
\[
\liminf_{t\to +\infty}\frac{|\phi_t(0)-\tau|}{|\tilde\phi_t(0)-\tilde\tau|}<+\infty,
\]
where $\tau\in \partial \D$ is the Denjoy-Wolff point of $(\phi_t)$ and $	\tilde\tau\in \partial \D$ is the Denjoy-Wolff point of $(\tilde\phi_t)$.
\end{theorem}

Also, we are able to provide a (complete) affirmative answer to the question in many cases:

\begin{theorem}\label{Thm:main2}
Let $(\phi_t), (\tilde \phi_t)$ be non-elliptic semigroups in $\D$. Let $h$ ({\sl respectively}, $\tilde h$) be the Koenigs function of $(\phi_t)$ ({\sl resp.} of $(\tilde \phi_t)$). Suppose that $h(\D)\subset \tilde h(\D)$ and that 
\begin{enumerate}
\item either $h(\D)$ is quasi-symmetric with respect to vertical axes,
\item or, $\tilde h(\D)$ is quasi-symmetric with respect to vertical axes,
\item or, $\tilde h(\D)$ is starlike with respect to some $w_0\in \tilde h(\D)$.
\end{enumerate}
 Then 
\[
\liminf_{t\to +\infty}[v^o(t)-\tilde v^o(t)]>-\infty,
\]
or, equivalently, there exists $K>0$ such that for all $t\geq 0$,
\[
|\phi_t(0)-\tau|\leq K |\tilde\phi_t(0)-\tilde \tau|,
\]
where $\tau\in \partial \D$ is the Denjoy-Wolff point of $(\phi_t)$ and $	\tilde\tau\in \partial \D$ is the Denjoy-Wolff point of $(\tilde\phi_t)$.
\end{theorem}
Here, we say that a starlike at infinity domain $\Omega$ is {\sl quasi-symmetric with respect to vertical axes} if there exists $K>0$ such that $K^{-1}\delta^-(t)\leq \delta^+(t)\leq K\delta^-(t)$, for all $t\geq 0$, where for some $z_0\in \Omega$, we denote by 
\begin{equation*}
\begin{split}
\tilde\delta^+(t)&:=\inf\{|w-(z_0+it)|: \Re w\geq \Re z_0, w\in \C\setminus \Omega\}, \\
\tilde\delta^-(t)&:=\inf\{|w-(z_0+it)|: \Re w\leq \Re z_0, w\in \C\setminus \Omega\},
\end{split}
\end{equation*}
and $\delta^\pm(t):=\min\{t, \tilde\delta^\pm(t)\}$. 

It was proved in \cite{BCGDZ} that $h(\D)$ is quasi-symmetric with respect to vertical axes if and only if $(\phi_t(0))$ converges non-tangentially to the Denjoy-Wolff point. 

Condition (3) in Theorem \ref{Thm:main2} is clearly satisfied by the sectors of type $W_{\alpha,\beta}$, with $\alpha, \beta\in [0,\pi]$ and $\alpha+\beta>0$, hence, our theorem generalizes the results in \cite{Bet} and \cite[Thm 5.3]{BetCD}.

In \cite[Thm. 5.6]{BetCD}, the authors get some estimates in the case where $h(\D)=W_{\alpha,\beta}$, with $\alpha, \beta\in [0,\pi]$ and $\alpha+\beta>0$. Indeed, they prove that, 
\begin{enumerate}
\item if $\alpha, \beta>0$, then there exists $K>0$ such that $|\tilde\phi_t(0)-\tau|\geq K t^{-1/(\alpha+\beta)}$ for all $t\geq 0$, 
\item if either $\alpha=0$ or $\beta=0$, then there exists $K>0$ such that $|\tilde\phi_t(0)-\tau|\geq K t^{-1-1/(\alpha+\beta)}$, for all $t\geq 0$.
\end{enumerate}
Now, if $\alpha, \beta>0$, then $h(\D)=W_{\alpha,\beta}$ is quasi-symmetric with respect to vertical axes and then the result can be obtained also from Theorem~\ref{Thm:main2} (and the explicit computation of the orthogonal speed of $(\phi_t)$). While, if either $\alpha=0$ or $\beta=0$, the picture does not enter into the hypotheses of Theorem~\ref{Thm:main2} because $h(\D)$ is not quasi-symmetric with respect to the vertical axes and we have no information on $\tilde h(\D)$. However, the estimate (2) in Theorem \ref{Thm:main2} is {\sl not} a relation between the orthogonal speeds of $(\phi_t)$ and $(\tilde\phi_t)$ (but between the orthogonal speeds of $(\tilde\phi_t)$ and the {\sl total speed} of $(\phi_t)$), and can be also obtained by the methods illustrated in this paper (see Remark~\ref{Rem:BetCD}).

The proof of Theorem~\ref{Thm:main2} is based on harmonic measure theory. Suppose $\Omega\subsetneq\C$ is a simply connected domain. The harmonic measure at a point $w\in \Omega$ with respect to $D\subset\partial \Omega$ is denoted by $\ha \left(w, D,\Omega \right)$. In Proposition~\ref{Prop:harmonic-measure-vo}, we prove that
there exists a constant $K>0$ such that for all $t\geq 1$,
\[
|v^0(t)+\frac{1}{2} \log  \ha \left(0, A_t , \D \right)|\leq K,
\]
where $A_t$ is defined as follows.
For $t\geq 1$, let $a_t\in\partial \D\cap \{\Im z>0\}$ be the intersection of $\partial \D$ with the circle containing $\overline{\tau}\phi_t(0)$, orthogonal to $(-1,1)$ and orthogonal to $\partial \D$ at $a_t$. Then let $\tilde A_t\subset \partial \D$ be the closed arc containing $1$ with end points $a_t$ and $\overline{a_t}$. Define $A_t:=\tau \tilde A_t$.

Then, in Section~\ref{Sec:estimates}, we give some estimates of harmonic measures, based on Gaier's Theorem and the Strong Markov Property. With these tools at hand, in the fundamental Lemma~\ref{Lemma:final}, we show the (almost) monotonicity of the orthogonal speed, in the case where a certain harmonic measure along the orbit of the semigroup is bounded from below by zero. This lemma allows us to prove Theorem~\ref{Thm:almost-asymp-mono} and Theorem~\ref{Thm:almost-asymp-mono1}, which is a more general version  of Theorem~\ref{Thm:main2} (and, from which, Theorem~\ref{Thm:main2} follows). In Section~\ref{sec:appl}, we give some applications of our results. In particular, we discuss the rate of convergence in case the image of the Koenigs function contains/is contained in domains of type  $\Pi_\alpha:=\{z\in \mathbb C: \Im z>|\Re z|^\alpha\}$ for $\alpha>1$ and of type $\Xi(\alpha,\theta):=(-\overline{\Ha}\cap\Pi_{\alpha})\cup W(\theta)$, where $W(\theta):=\left\lbrace z\in \C\mid\arg (z)\in\left(\frac{\pi}{2}-\theta,\frac{\pi}{2}\right) \right\rbrace$. 

We end the paper with Section~\ref{sec:open} containing some open questions originating from this work.

\medskip

We thank the referee for several comments which improved the original manuscript.

\section{Semigroups in the unit disc}

In this section we briefly recall the basics of the theory of semigroups of holomorphic self-maps of the unit disc, as needed for our aims. We refer the reader to the books \cite{Ababook89, BCDbook, EliShobook10, Shobook01} for details.

\begin{definition}
A {\sl continuous semigroup $(\phi_t)$ of holomorphic self-maps of $\D$}, or just a {\sl semigroup in $\D$} for short, is a semigroup homeomorphism between the semigroup of real non-negative numbers (with respect to  sum) and the semigroup of holomorphic self-maps of $\D$ (with respect to composition). Here, as usual,  the chosen topology for  $\R^+$  is the Euclidean topology and the space of holomorphic self-maps of $\D$ is endowed with the topology of uniform convergence on compacta. \end{definition}

A semigroup $(\phi_t)$ without fixed points in $\D$ is called {\sl non-elliptic}. If $(\phi_t)$ is a non-elliptic semigroup, $\phi_t$ has the same {\sl Denjoy-Wolff point} $\tau\in \partial \D$, for all $t>0$. Moreover, $\lim_{t\to+\infty}\phi_t(z)=\tau\in \partial \D$, for all $z\in \D$. 

Let $(\phi_t)$ be a non-elliptic semigroup in $\D$. Up to conjugate with a rotation, we can assume that the Denjoy-Wolff point of $(\phi_t)$ is $1$. The Denjoy-Wolff Theorem (see, {\sl e.g.} \cite[Thm. 1.8.4]{BCDbook}) implies that 
\begin{equation}\label{Eq:DW}
\phi_t(E(1,R))\subseteq E(1,R),
\end{equation}
for all $t\geq 1$ and $R>0$, where $E(1,R):=\{z\in \D: |1-z|^2<R (1-|z|^2)\}$. 
 
Let us denote the right half-plane by
\[
\Ha:=\{w\in \C: \Re w>0\}
\]
and let $C:\D \to \Ha$ be the Cayley transform defined by $C(z)=(1+z)/(1-z)$. Then \eqref{Eq:DW} implies that for all $s\geq t\geq 0$,
\begin{equation}\label{Eq:DW-Ha}
 \Re (C(\phi_s(0)))\geq \Re (C(\phi_t(0))).
 \end{equation}
This is, in fact, the Denjoy-Wolff Theorem version in $\Ha$ (see also \cite[Thm. 1.7.8]{BCDbook}).

If $(\phi_t)$ is a non-elliptic semigroup in $\D$, then there exists a (essentially unique) univalent function $h:\D \to \C$ such that
\begin{enumerate}
\item $h(\phi_t(z))=h(z)+it$ for all $z\in \D$, $t\geq 0$,
\item $\bigcup_{t\geq 0}(h(\D)-it)=\Omega$, where $\Omega$ is either a vertical strip, or a vertical half-plane or $\C$.
\end{enumerate}
The function $h$ is called the {\sl Koenigs function} of $(\phi_t)$.


\section{Speeds of semigroups}

Speeds of non-elliptic semigroups in $\D$ have been introduced in \cite{Br} (see also \cite[Ch. 16]{BCDbook}). We recall here the basic facts needed.

Let $\tau\in \partial \D$ and let $\Gamma:=(-1,1)\tau$. Then $\Gamma$ is a geodesic for the hyperbolic distance $k_\D$ in $\D$. For every $z\in \D$, there exists a  unique point, $\pi(z)\in \Gamma$ such that
\[
k_\D(z,\pi(z))=\min\{k_\D(z,w): w\in \Gamma\}.
\]

\begin{definition}
Let $(\phi_t)$ be a non-elliptic semigroup in $\D$ with Denjoy-Wolff point $\tau\in \partial \D$. The {\sl (total) speed} $v(t)$ of $(\phi_t)$ is
\[
v(t):=k_\D(0,\phi_t(0)), \quad t\geq 0.
\]
The {\sl orthogonal speed} $v^o(t)$ of $(\phi_t)$ is
\[
v^o(t):=k_\D(0,\pi(\phi_t(0))), \quad t\geq 0.
\]
The {\sl tangential speed} $v^T(t)$ of $(\phi_t)$ is
\[
v^T(t):=k_\D(\phi_t(0),\pi(\phi_t(0))), \quad t\geq 0.
\]
\end{definition}

As a consequence of  the ``Pythagoras' Theorem in hyperbolic geometry'', we have the following relation for all $t\geq 0$ (see \cite[eq. (16.1.2)]{BCDbook} or \cite[eq. (5.2)]{Br})
\begin{equation}\label{Pyt}
v^o(t)+v^T(t)-\frac{1}{2}\log 2\leq v(t)\leq v^o(t)+v^T(t).
\end{equation}
Also, as a consequence of the Julia's Lemma and \eqref{Pyt} (see, \cite[eq. (16.1.3)]{BCDbook} or \cite[eq. (5.3)]{Br})
\begin{equation}\label{dip-speed}
v^T(t)\leq v^o(t)+4\log 2.
\end{equation}

Moreover, the speeds of a semigroup are related to certain quantities, whose asymptotic estimates go under the name ``rate of convergence'' of a semigroup. For all $t\geq 0$, we have
\begin{equation}\label{Eq:iper-eucl}
\begin{split}
&\left|v(t)-\frac{1}{2} \log\frac{1}{1-|\phi_t(0)|}\right|\leq \frac{1}{2}\log 2,\\
&\left|v^o(t)-\frac{1}{2} \log\frac{1}{|\tau-\phi_t(0)|}\right|\leq \frac{1}{2}\log 2,\\
&\left|v^T(t)-\frac{1}{2} \log\frac{|\tau-\phi_t(0)|}{1-|\phi_t(0)|}\right|\leq \frac{3}{2}\log 2.
\end{split}
\end{equation}

Since the definition of the speeds is given in hyperbolic terms, the speeds are invariant under conformal changes of coordinates. In particular, one can check that if $(\phi_t)$ is a non-elliptic semigroup in $\D$ with Denjoy-Wolff point $1$ and $C:\D \to \Ha$ is the Cayley transform $C(z)=(1+z)/(1-z)$, then (see \cite[eq. (5.1)]{Br} or \cite[Sec. 6.5]{BCDbook}) the orthogonal speed of $(\phi_t)$ is  
\begin{equation}\label{Eq:ort-speed-inH}
v^{o}(t)= k_{\Ha} (1, \rho_t)=\frac{1}{2}\log \rho_t,
\end{equation}
where we let $C(\phi_t(0))=\rho_t e^{i\theta_t}$ for some $\rho_t>0$ and $\theta_t\in (-\frac{\pi}{2}, \frac{\pi}{2})$, $t\geq 0$.

In particular, by \eqref{Pyt} and \eqref{dip-speed}, we have 
\[
v(t)\leq 2 v^o(t)+4\log 2.
\]
Since $v(t)\to +\infty$, as $t\to +\infty$, (because $\phi_t(0)\to \tau\in \partial \D$), it follows that $\lim_{t\to+\infty}v^o(t)=+\infty$ and, in particular,  $\lim_{t\to +\infty}\rho_t=+\infty$.

\begin{remark}\label{Rem:BetCD}
Let $(\phi_t), (\tilde \phi_t)$ be non-elliptic semigroups in $\D$. Let $h$ ({\sl respectively}, $\tilde h$) be the Koenigs function of $(\phi_t)$ ({\sl resp.} of $(\tilde \phi_t)$). Let $v(t), v^o(t)$ and $\tilde v(t), \tilde v^o(t)$ denote the total and the orthogonal speeds of $(\phi_t)$ and $(\tilde\phi_t)$, respectively. Suppose that $h(\D)\subset \tilde h(\D)$. Then clearly $v(t)\leq \tilde v(t)$. Moreover, by \eqref{Pyt},
\[
\tilde v^o(t)\leq \tilde v(t)+\frac{1}{2}\log 2\leq  v(t)+\frac{1}{2}.
\] 
Hence,
\[
\liminf_{t\to +\infty} [v(t)-\tilde v^o(t)]>-\infty.
\]
For instance, if $h(\D)=W_{\alpha,\beta}$,  with $\alpha, \beta\in [0,\pi]$ and $\alpha+\beta>0$, but either $\alpha=0$ or $\beta=0$, then $v(t)\sim \frac{\pi+\alpha+\beta}{2(\alpha+\beta)}\log t$ (see \cite[Prop. 6.5]{Br} or \cite[Cor. 16.2.6]{BCDbook}). From this, condition (2) of \cite[Thm. 5.6]{BetCD} follows. 

It is presently unknown if we can replace $\frac{\pi+\alpha+\beta}{2(\alpha+\beta)}\log t$ with (the more natural) estimate $v^o(t)\sim \frac{\pi}{2(\alpha+\beta)}\log t$.
\end{remark}

In the final part of this section, we give some geometric conditions on the image of the Koenigs function of a semigroup, which assures that $v^o(t)$ is a non-decreasing function of $t$.

\begin{lemma}\label{Lem:non-dec-tot-ort}
Let $(\phi_t)$ be a non-elliptic semigroup in $\D$. Suppose that $v(t_2)\geq v(t_1)$, for some $t_2\geq t_1\geq 0$. Then $v^o(t_2)\geq v^o(t_1)$. \end{lemma}

\begin{proof}
Suppose $t_2>t_1$ and $v(t_2)\geq v(t_1)$. We can assume that the Denjoy-Wolff point of $(\phi_t)$ is $1$. Let $C(z):=\frac{1+z}{1-z}$ be the Cayley transform from $\D$ to $\Ha$. Let $\rho_t e^{i\theta_t}:=C(\phi_t(0))$,  with $\rho_t>0$ and $\theta_t\in (-\pi/2,\pi/2)$, $t\geq 0$. Then, $v(t)=k_\Ha(1, \rho_t e^{i\theta_t})$. 
By \eqref{Eq:DW-Ha}, 
\begin{equation}\label{Eq:JWC1H}
\rho_{t_2} \cos \theta_{t_2} \geq \rho_{t_1} \cos\theta_{t_1}\geq 1. 
\end{equation}
This implies that $\rho_{t_2}e^{i\theta_{t_2}}$ belongs to the set $\{w\in \C: \Re w\geq \rho_{t_1} \cos\theta_{t_1}\}$. 

Let $D(1, v(t_1)):=\{w\in \Ha: k_\Ha(1, w)<v(t_1)\}$, which is a Euclidean disc of center a real number $r\in (0,+\infty)$, containing $1$ in its interior and $\rho_{t_1}e^{i\theta_{t_1}}$ on its boundary (in fact, the center is $\cosh (2v(t_1))$ and the radius $\sinh (2 v(t_1))=|\cosh (2v(t_1))-\rho_{t_1}e^{i\theta_{t_1}}|$, but we do not need this explicit computation). In particular, $\partial D(1, v(t_1))$ contains both $\rho_{t_1}e^{i\theta_{t_1}}$ and $\rho_{t_1}e^{-i\theta_{t_1}}$. From this, a simple geometric consideration shows that 
\[
\{w\in \C: \Re w\geq \rho_{t_1} \cos\theta_{t_1}, |w|\leq \rho_{t_1}\}\subset D(1, v(t_1)).
\]
From the hypothesis, since $\rho_{t_2}e^{i\theta_{t_2}}\not \in D(1, v(t_1))$, the previous equation together with \eqref{Eq:JWC1H} implies immediately that $\rho_{t_2}\geq \rho_{t_1}$, and, hence, $v^o(t_2)\geq v^o(t_1)$.
\end{proof}

\begin{proposition}\label{Prop:non-decr-v}
Let $(\phi_t)$ be a non-elliptic semigroup with Koenigs function $h$. If $h(\D)$ is convex, then $[0,+\infty)\ni t \mapsto v(t)$ is non-decreasing.
\end{proposition}
\begin{proof}
Let $0\leq t_1\leq t_2$ and assume by contradiction that $v(t_2)<v(t_1)$.  
Note that $v(t)=k_\D(0,\phi_t(0))=k_{h(\D)}(h(0), h(0)+it)$. Hence, if $v(t_2)<v(t_1)$, it follows that  $h(0)+it_2\in D(h(0), v(t_1)):=\{w\in \C: k_{h(\D)}(h(0), w)<v(t_1)\}$.
Since the hyperbolic distance in a convex domain is a convex function, it follows that the hyperbolic discs are convex. 
Therefore, if $h(0)+it_2\in D(h(0), v(t_1))$, since $h(0)\in D(h(0), v(t_1))$ as well, it follows that $h(0)+is\in D(h(0), v(t_1))$ for all $s\in [0, t_2]$. However, $h(0)+it_1\in D(h(0), v(t_1))$ and equivalently, $v(t_1)=k_{h(\D)}(h(0), h(0)+it_1)<v(t_1)$. We are led to a contradiction. 
\end{proof}

More generally, we have the following result. 

\begin{proposition}\label{Prop:non-decr-v-st}
Let $(\phi_t)$ be a non-elliptic semigroup with Koenigs function $h$. If $h(\D)$ is starlike with respect to  $h(0)$ then $[0,+\infty)\ni t \mapsto v(t)$ is non-decreasing.
\end{proposition}
\begin{proof}
Since the hyperbolic discs centered at $h(0)$ are also starlike with respect to $h(0)$ (see, {\sl e.g.}, \cite[Thm. 2.10]{Du}), the proof is similar to the proof of Proposition~\ref{Prop:non-decr-v} and we omit it.
\end{proof}


\section{Orthogonal Speed and Harmonic Measure}

\begin{lemma}\label{lem1}
Let $(\phi_t)$ be a non-elliptic semigroup in $\D$ with Denjoy-Wolff point $1$. Let $C(z):=\frac{1+z}{1-z}$ be the Cayley transform from $\D$ to $\Ha$. Let $\rho_t e^{i\theta_t}:=C(\phi_t(0))$,  with $\rho_t>0$ and $\theta_t\in (-\pi/2,\pi/2)$, $t\geq 0$. There exists $K>0$ such that for all $t\geq 1$, 
\[
\left| v^{o}(t) + \frac{1}{2} \log  \ha \left(1, \Theta_t , \Ha \right) \right| \leq K,
\]
where  $\Theta_t:= \{iy : |y| \geq \rho_t \}=\{iy : |y| \geq |C(\phi_t(0))| \}$.
\end{lemma} 
\begin{proof}
Let $\ha_t:= \ha (1, \Theta_t, \Ha)$. By \cite[Example 7.2.5]{BCDbook}, 
\[
\ha_t=\frac{1}{\pi}\Arg\left( \frac{i\rho_t -1}{1+i\rho_t}\right).
\]
Since $\lim_{t\to+\infty}\rho_t=+\infty$, there exists $t_0> 0$ such that $\rho_t>1$, for all $t\geq t_0$. Hence, from the previous formula, for all $t\geq t_0$,
\[
\ha_t=\frac{1}{\pi}\arctan\frac{2\rho_t}{\rho_t^2-1}.
\]	
Moreover, there exists $t_1\geq t_0$ such that $\frac{2\rho_t}{\rho_t^2-1}<1$, for all $t\geq t_1$. For $y\in [0,1]$, we know that $\frac{\pi}{4} y \leq \arctan y \leq y$. Hence, there exist constants $0<c_1<c_2$ such that for all $t\geq t_1$,
\begin{equation*}
\frac{c_1}{\rho_t}\leq \frac{1}{2}\frac{1}{\rho_t-\frac{1}{\rho_t}}\leq \omega_t \leq \frac{2}{\pi}\frac{1}{\rho_t-\frac{1}{\rho_t}}\leq \frac{c_2}{\rho_t}.
\end{equation*}	
The above inequality and \eqref{Eq:ort-speed-inH} lead to the result at once.
\end{proof}

Let $(\phi_t)$ be a non-elliptic semigroup in $\D$ with Denjoy-Wolff point $1$. Note that, by the Denjoy-Wolff Theorem (see, {\sl e.g.}, \cite[Thm. 1.8.4]{BCDbook}), for every $t>0$, $\Re \phi_t(0)>0$. Bearing this in mind, we can state the following Proposition. 

\begin{proposition}\label{Prop:harmonic-measure-vo}
Let $(\phi_t)$ be a non-elliptic semigroup in $\D$ with Denjoy-Wolff point $1$. For $t\geq 1$, let $a_t\in\partial \D\cap \{\Im z>0\}$ be the intersection of $\partial \D$ with the circle containing $\phi_t(0)$, orthogonal to $(-1,1)$ and orthogonal to $\partial \D$ at $a_t$. Let $A_t\subset \partial \D$ be the closed arc containing $1$ and with end points $a_t$ and $\overline{a_t}$. Then there exists a constant $K>0$ such that
\[
\left| v^0(t)+\frac{1}{2} \log  \ha \left(0, A_t , \D \right) \right|\leq K, \quad \text{for all } t\geq 1.
\]
\end{proposition}
\begin{proof}
It follows at once from Lemma~\ref{lem1} and the conformal invariance of the harmonic measure under the Cayley transform.
\end{proof}

\section{Estimates of harmonic measures}\label{Sec:estimates}

In all this section, $(\phi_t)$ denotes a non-elliptic semigroup in $\D$ with Denjoy-Wolff point $1$. 

Let $C(z):=\frac{1+z}{1-z}$ be the Cayley transform from $\D$ to $\Ha$. Let $\rho_t e^{i\theta_t}:=C(\phi_t(0))$,  with $\rho_t>0$ and $\theta_t\in (-\pi/2,\pi/2)$, $t\geq 0$.
For $t\geq 1$, let
\[
\Gamma_t:=\{\rho_s e^{i\theta_s}: s\geq t\} \quad \text{and} \quad  \Gamma^\ast_t:=\{iy: |y|\geq \min_{s\geq t}\rho_s\}.
\]
In addition, set $\Theta_t:=\{iy: |y|\geq \rho_t\}$ and note that $\Gamma^\ast_t=\Theta_t$ if and only if $\rho_s\geq \rho_t$ for all $s\geq t$.

\begin{lemma}\label{lem:good-sequence}
There exists an increasing sequence $\{t_n\}$, with $t_1\geq 1$, converging to $+\infty$ such that $\Theta_{t_n}=\Gamma^\ast_{t_n}$, for all $n$.
\end{lemma}
\begin{proof}
Since $[0,+\infty)\ni t\mapsto \rho_t$ is continuous and $\lim_{t\to+\infty}\rho_t=+\infty$, there exists $t_1\geq 1$ so that $\rho_s\geq \rho_{t_1}$ for all $s\geq t_1$. Then, by induction, we take $t_n\geq t_{n-1}+1$ to be a point of minimum of $[t_{n-1}+1, +\infty)\ni t\mapsto \rho_t$.
\end{proof}

\begin{lemma}\label{lem:good-estim-increase}
Let $t\geq 1$. For all $s\geq t$,
\[
\ha(\rho_s e^{i\theta_s}, \Gamma_t^\ast, \Ha)\geq \frac{1}{2}.
\]
\end{lemma}
\begin{proof}
Let $t_0\geq t$ be such that $\rho_{t_0}:=\min\{\rho_s: s\geq t\}$. By definition, $\Gamma_t^\ast=\{iy: |y|\geq \rho_{t_0}\}$. Consider the automorphism $T:\Ha\to \Ha$ give by $T(w):=\frac{w}{\rho_{t_0}}$. Let $G_1:=\{iy: |y|\geq 1\}$. By the conformal invariance of harmonic measure, we have
\[
\ha (\rho_s e^{i\theta_s}, \Gamma_t^\ast, \Ha)=\ha\left(\frac{\rho_s}{\rho_{t_0}}e^{i\theta_s}, G, \Ha\right).
\]
Since $\frac{\rho_s}{\rho_{t_0}}\geq 1$, we have $\ha(\frac{\rho_s}{\rho_{t_0}}e^{i\theta_s}, G, \Ha)\geq 1/2$ (this follows from a direct computation, or see \cite[Lemma 7.1.10]{BCDbook} and use conformal invariance under the Cayley transform).
\end{proof}

\begin{lemma}\label{Lemma:non-tg}
Fix $\theta\in (0,\pi/2)$. Then there exists $C=C(\theta)>0$ such that for all $t\geq 1$, with $|\theta_t|\leq \theta$, we have
\[
\ha(\rho_s e^{i\theta_s}, \Theta_t, \Ha)\geq C,
\]
for all $s\geq t$.
\end{lemma}
\begin{proof}
By \eqref{Eq:DW-Ha}, for every $s\geq t$, $\rho_s \cos \theta_s \geq \rho_t \cos\theta_t$. 
Therefore, $\frac{\rho_s}{\rho_t}e^{i\theta_s}\in \{w\in \Ha: \Re w>\cos \theta_t\}$. Hence, repeating the argument in Lemma~\ref{lem:good-estim-increase} with $t_0=t$, we obtain 
\[
\ha (\rho_s e^{i\theta_s}, \Theta_t, \Ha)=\ha\left(\frac{\rho_s}{\rho_{t}}e^{i\theta_s}, G, \Ha\right)>C(\theta)>0,
\]
where 
\[
C(\theta):=\min\{\ha\left(w, G, \Ha\right) : \Re w>\cos \theta\}.
\]
\end{proof}

\begin{lemma}\label{Lem:Hall}
For all $t\geq 1$,
\[
 \ha (1, \Theta_t, \Ha)<2 \ha(1, \Gamma_t, \Ha\setminus \Gamma_t).
\]
\end{lemma}
\begin{proof}
This is essentially a consequence of the  Hall's (or Gaier's) Theorem. To give some details, let $a_t\in \partial \D$ and $A_t\subset \partial \D$ be as in Proposition~\ref{Prop:harmonic-measure-vo}. Let $A'_t\subset A_t$ be the arc with end points $1$ and $a_t$. Let $W_t:=\{\phi_s(0): s\geq t\}$, $t\geq 1$. Then, by Gaier's Theorem (see, {\sl e.g.}, \cite[Thm. 7.2.13]{BCDbook}), for all $t\geq 1$,
\[
\ha(0, W_t, \Ha\setminus W_t)> \ha(0, A'_t, \Ha).
\]
Now, by definition of harmonic measure (or see, {\sl e.g.} \cite[Eq. (7.1.2)]{BCDbook}), denoting by $\ell(A'_t)$ the Euclidean length of $A'_t$, we have
\[
\ha(0, A'_t, \Ha)=\frac{1}{2\pi}\ell(A'_t)=\frac{1}{2\pi}\frac{\ell(A_t)}{2}=\frac{1}{2}\ha(0, A_t, \Ha).
\]
Therefore, $\ha(0, W_t, \Ha\setminus W_t)>\frac{1}{2}\ha(0, A_t, \Ha)$. Using the conformal invariance of the harmonic measure and the Cayley transform, we have the result.
\end{proof}

\begin{lemma}\label{Lem:reverse}
Let $t\geq 1$. Suppose that there exists $c=c(t)>0$ such that for all $s\geq t$,
\begin{equation}\label{Eq:reverse}
\ha(\rho_s e^{i\theta_s}, \Theta_t, \Ha)\geq c.
\end{equation}

Then 
\[
\omega(1,\Theta_t,\Ha)\geq c\omega(1,\Gamma_t, \Ha\setminus\Gamma_t).
\]
\end{lemma}
\begin{proof}
By the Strong Markov Property for harmonic measure (see, \cite[Lemma 3.7]{Bet98}), we have
\[
\ha(1,\Theta_t, \Ha)=\ha(1, \Theta_t, \Ha\setminus\Gamma_t)+\int_{\Gamma_t}\ha(\alpha, \Theta_t, \Ha)\ha(1, d\alpha, \Ha\setminus\Gamma_t),
\]
where, considering the measure $\lambda:=\ha(1, \cdot , \Ha\setminus\Gamma_t)$ on the boundary of $\Ha\setminus\Gamma_t$,  we let $\ha(1, d\alpha, \Ha\setminus\Gamma_t):=d\lambda$ ({\sl i.e.}, the integration with respect to the measure $\lambda$). 

Therefore, by hypothesis \eqref{Eq:reverse}, 
\begin{equation*}
\begin{split}
\ha(1,\Theta_t, \Ha)&\geq \int_{\Gamma_t}\ha(\alpha, \Theta_t, \Ha)\ha(1, d\alpha, \Ha\setminus\Gamma_t)\\&\geq c\int_{\Gamma_t}\ha(1, d\alpha, \Ha\setminus\Gamma_t)=\omega(1,\Gamma_t, \Ha\setminus\Gamma_t).
\end{split}
\end{equation*} 
\end{proof}


\section{Asymptotic monotonicity of orthogonal speed}

In this section, $(\phi_t)$ and $(\tilde\phi_t)$ are non-elliptic semigroups in $\D$ with Koenigs functions $h$ and $\tilde h$, respectively. We assume that $1$ is the Denjoy-Wolff point of both $(\phi_t)$ and $(\tilde \phi_t)$.

We use the notations introduced in the previous section, and we let $\Gamma_t, \Gamma_t^\ast, \Theta_t$ be the sets associated to $\phi_t$ and $\tilde\Gamma_t, \tilde\Gamma_t^\ast, \tilde\Theta_t$ the corresponding ones associated to $(\tilde\phi_t)$.

\begin{lemma}\label{Lemma:final}
Suppose $h(\D)\subset \tilde h(\D)$. Let $c>0$. Then there exists a constant $H\in \R$ such that, for every $t\geq 1$ so that 
\begin{equation}\label{Eq:Lemma:final}
 \ha(\tilde\rho_s e^{i\tilde\theta_s}, \tilde\Theta_t, \Ha)\geq c \quad \forall s\geq t,
\end{equation}
we have
\[
v^o(t)-\tilde v^o(t) \geq H.
\]
\end{lemma}
\begin{proof}
Let $C:\D \to \Ha$ be the Cayley transform given by $C(w)=(1+w)/(1-w)$. Hence, $h\circ C^{-1}:\Ha \to h(\D)$ is a biholomorphism such that $h(0)=r+it_0$, for some $r, t_0\in \R$ and  $h(C^{-1}(\Gamma_t))=r+i[t_0+t,+\infty)$. Similarly,  $\tilde h\circ C^{-1}:\Ha \to \tilde h(\D)$ is a biholomorphism mapping $\tilde\Gamma_t$ onto $\tilde{r}+i[\tilde t_0+t,+\infty)$, with $\tilde h(0)=\tilde{r}+i\tilde t_0$, for some $\tilde r, \tilde t_0\in \R$.
\smallskip

\noindent {\sl Case 1. Assume $r=\tilde r$ and $t_0=\tilde t_0$}.

\smallskip

Let $T:=r+i[t_0+t,+\infty)$. By (in order of usage) Lemma~\ref{Lem:Hall}, conformal invariance, domain monotonicity and again conformal invariance, we obtain
\begin{equation*}
\begin{split}
\ha(1,  \Theta_t, \Ha)&\stackrel{\tiny \hbox{(Lemma}~\ref{Lem:Hall})}{<} 2 \ha(1,\Gamma_t, \Ha\setminus\Gamma_t)\stackrel{\tiny \hbox{(conformal inv.)}}{=}2 \ha(r+i t_0, T,  h(\D)\setminus T)\\& \stackrel{\tiny \hbox{(domain monoton.)}}{\leq} 2 \ha(r+i t_0,T,  \tilde h(\D)\setminus T)\stackrel{\tiny \hbox{(conformal inv.)}}{=}2\ha (1,\tilde \Gamma_t, \Ha\setminus \tilde \Gamma_t)\\& \stackrel{\tiny \hbox{(Lemma}~\ref{Lem:reverse})}{\leq}\frac{2}{c}\ha(1,  \tilde\Theta_t, \Ha).
\end{split}
\end{equation*}
Therefore, by Lemma~\ref{lem1} (denoting by $\tilde K>0$ the constant related to $(\tilde\phi_t)$ and by $K>0$ the one related to $(\phi_t)$), we have
\begin{equation*}
\begin{split}
v^o(t)&\geq -\frac{1}{2}\log \ha(1,\Theta_t, \Ha)- K\geq -\frac{1}{2}\log \ha(1,\tilde\Theta_t, \Ha)- K-\frac{1}{2}\log \frac{2}{c}\\&\geq \tilde v^o(t)+\tilde K- K-\frac{1}{2}\log \frac{2}{c}.
\end{split}
\end{equation*}
Setting $H:=\tilde K- K-\frac{1}{2}\log \frac{2}{c}$, we have the result in this case.
\smallskip

\noindent {\sl Case 2. General case}. 

\smallskip

Let $w_0\in \D$ be such that $\tilde h(w_0)=r+it_0$ (this is possible because $h(\D)\subset \tilde h(\D)$). Let $A:\D \to \D$ be an automorphism such that $A(1)=1$ and $A(w_0)=0$. Let $\tilde\varphi_t:=A\circ \tilde\phi_t\circ A^{-1}$. Hence, $(\tilde\varphi_t)$ is a non-elliptic semigroup in $\D$ with Denjoy-Wolff point $1$, and it is easy to check that $\tilde h\circ A^{-1}$ is the Koenigs function of $(\tilde\varphi_t)$. Moreover, $\tilde h\circ A^{-1}(0)=\tilde h(w_0)=r+it_0$. Therefore, by Case 1, 
\[
v^o(t)-\tilde w^0(t)\geq H,
\]
where $\tilde w^0(t)$ denotes the orthogonal speed of $(\tilde\varphi_t)$. By \cite[Prop. 16.1.6]{BCDbook}, there exists $H'>0$ such that $|\tilde v^o(s)-\tilde w^o(s)|\leq H'$ for all $s\geq 0$, hence 
\[
v^o(t)-\tilde v^o(t)\geq H-H'.
\]
\end{proof}
 
\begin{proof}[Proof of Theorem~\ref{Thm:almost-asymp-mono}]
By \cite[Prop. 16.1.6]{BCDbook}, up to conjugation, we can assume without loss of generality that $1$ is the Denjoy-Wolff point of both $(\phi_t)$ and $(\tilde \phi_t)$.

By Lemma~\ref{lem:good-sequence} and Lemma~\ref{lem:good-estim-increase}, there exists an increasing sequence $\{t_n\}$, $t_1\geq 1$, converging to $+\infty$ such that $\ha(\tilde\rho_s e^{i\tilde\theta_s}, \tilde\Theta_{t_n}, \Ha)\geq 1/2$, for all $s\geq t_n$. Therefore, by Lemma~\ref{Lemma:final}, there exists $H\in \R$ such that 
\[
v^o(t_n)-\tilde v^o(t_n)\geq H
\]
for all $n$. The wanted statement follows at once from \eqref{Eq:iper-eucl}.
\end{proof}

\begin{theorem}\label{Thm:almost-asymp-mono1}
Let $(\phi_t), (\tilde \phi_t)$ be non-elliptic semigroups in $\D$. Let $h$ ({\sl respectively}, $\tilde h$) be the Koenigs function of $(\phi_t)$ ({\sl resp.} of $(\tilde \phi_t)$). Suppose that $h(\D)\subset \tilde h(\D)$ and that 
\begin{enumerate}
\item either $\{\phi_t(0)\}$ converges non-tangentially to the Denjoy-Wolff point,
\item or, $\{\tilde\phi_t(0)\}$ converges non-tangentially to the Denjoy-Wolff point,
\item or, $[0,+\infty)\ni t\mapsto \tilde v^o(t)$ is (eventually) non-decreasing,
\item or, $[0,+\infty)\ni t\mapsto \tilde v(t)$ is (eventually) non-decreasing.
\end{enumerate}
 Then 
\[
\liminf_{t\to +\infty}[v^o(t)-\tilde v^o(t)]>-\infty.
\]
\end{theorem}
\begin{proof}
By \cite[Prop. 16.1.6]{BCDbook}, up to conjugation, we can assume without loss of generality that $1$ is the Denjoy-Wolff point of both $(\phi_t)$ and $(\tilde \phi_t)$.

(1) In this hypothesis, $\limsup_{t\to+\infty}v^T(t)<+\infty$, hence, by \eqref{Pyt}, there exists $c_1>0$ such that $|v(t)-v^o(t)|\leq c_1$, for all $t\geq 1$. 

Since $h(\D)\subset \tilde h(\D)$, then $v(t)\geq \tilde v(t)+c_2$, for some $c_2\in \R$ and for all $t\geq 0$. Taking into account again \eqref{Pyt}, we have
\[
v^o(t)\geq v(t)+c_1\geq \tilde v(t)+c_1+c_2\geq \tilde v^o(t)+c_1+c_2-\frac{1}{2}\log 2,
\]
for all $t\geq 0$, and we are done.

(2) In this hypothesis, by Lemma~\ref{Lemma:non-tg}, there exists $C>0$ such that for all $t\geq 1$, we have $\ha(\tilde\rho_s e^{i\tilde\theta_s}, \tilde\Theta_{t}, \Ha)\geq C$, for all $s\geq t$. Therefore, by Lemma~\ref{Lemma:final}, there exists $H\in \R$ such that 
\[
v^o(t)-\tilde v^o(t)\geq H,
\]
for all $t\geq 1$.

(3) The map $t\mapsto \tilde v^o(t)$ is (eventually) non-decreasing if and only if $t\mapsto \frac{1}{2}\log \tilde\rho_r$ is (eventually) non-decreasing, if and only if $t\mapsto  \tilde\rho_r$ is (eventually) non-decreasing. By definition, the latter condition is eventually equivalent to $\tilde\Gamma_t^\ast=\tilde\Theta_t$. If this is satisfied, by Lemma~\ref{lem:good-estim-increase}, $\ha(\tilde\rho_s e^{i\tilde\theta_s}, \tilde\Theta_{t}, \Ha)\geq 1/2$, for all $s\geq t$ and for all $t$ large enough. Again, the result follows then from Lemma~\ref{Lemma:final}.

(4) It follows at once from Lemma~\ref{Lem:non-dec-tot-ort} and (3).
\end{proof}

\begin{proof}[Proof of Theorem~\ref{Thm:main2}]
(1) ({\sl respectively} (2)) follows at once by Theorem~\ref{Thm:almost-asymp-mono1} (1) ({\sl resp.  (2)}) and \cite[Thm. 1.1]{BCGDZ} (or \cite[Thm. 17.3.1]{BCDbook}).

(3) In case $w_0=\tilde h(0)$,  the result follows from  Proposition~\ref{Prop:non-decr-v-st} and Theorem~\ref{Thm:almost-asymp-mono1}.(4).

In case $w_0\neq \tilde h(0)$, let $A:\D\to \D$ be an automorphism of $\D$ such that $A(w_0)=0$. Let $\tilde\varphi_t:=A\circ \tilde\phi_t\circ A^{-1}$. Hence, $(\tilde\varphi_t)$ is a non-elliptic semigroup in $\D$, and it is easy to check that $h_1:=\tilde h\circ A^{-1}$ is the Koenigs function of $(\tilde\varphi_t)$. Since $h_1(\D)$ is starlike with respect to $0$ by construction and hypothesis, it follows by Proposition~\ref{Prop:non-decr-v-st} that the total speed $w(t)$ of $(\tilde\varphi_t)$ is non-decreasing. Hence, by Theorem~\ref{Thm:almost-asymp-mono1}.(4), 
\[
\liminf_{t\to +\infty}[v^o(t)-w^o(t)]>-\infty,
\]
where $w^o(t)$ denotes the orthogonal speed of $(\tilde\varphi_t)$. By \cite[Prop. 16.1.6]{BCDbook}, there exists a constant $K_1>0$ such that $|\tilde v^o(t)-w^o(t)|\leq K_1$ for all $t\geq 0$. 
The wanted statement follows at once from \eqref{Eq:iper-eucl}.
\end{proof}

\section{Some applications}\label{sec:appl}

As it is clear from Theorem \ref{Thm:almost-asymp-mono} or Theorem~\ref{Thm:main2} (and \eqref{Eq:iper-eucl}), in order to obtain explicit estimates for the rate of convergence of orbits in terms of the geometry of the image of the Koenigs function of a semigroup, the main issue is to have estimates of the rate of convergence in special domains.  

In this section, we estimate the orthogonal speed of semigroups whose Koenigs function has image given by some special forms and apply our main results to get general applications.

\bigskip

\noindent{\bf 1.} 	Fix $\alpha>1$ and consider the following simply connected domain
	\[
	\Pi_{\alpha}:=\{ z\in\mathbb{C}\mid \Im z>|\Re z|^{\alpha}\}
	\]
The domain $\Pi_{\alpha}$ is starlike at infinity. Therefore, if $h_\alpha:\D\to \Pi_\alpha$ is a Riemann map, it turns out that $h_\alpha$ is the Koenigs function of the semigroup $(\phi^\alpha_t)$ where $\phi^\alpha_t(z):=h^{-1}(h(z)+it)$, $z\in \D$, $t\geq 0$. 	
\begin{figure}[htb]
		\centering
		\begin{tikzpicture}[scale=0.8]
		\draw (0,0) parabola (2,4);
		\draw (2.02,4.1) parabola (2.05,4.2);
		\draw (0,0) parabola (-2,4);
		\draw (-2.02,4.1) parabola (-2.05,4.2);
		\draw[thick, ->] (0,1) node[left]{$i$}--(0,3.94) node[left]{$it$}--(0,4.2);
		\draw[thick,dashed] (0,3.74)--(1.8,3.24) node[right]{$\sqrt{t-\frac{1}{2}}+i\left(t-\frac{1}{2}\right)$};
		\draw[thick,dashed] (0,3.74)--(-1.8,3.24) node[left]{$-\sqrt{s-\frac{1}{2}}+i\left(t-\frac{1}{2}\right)$};
		\end{tikzpicture}
		\caption[Parabola]{The domain $\Pi_{\alpha}$ with $\alpha=2$}
	\end{figure}
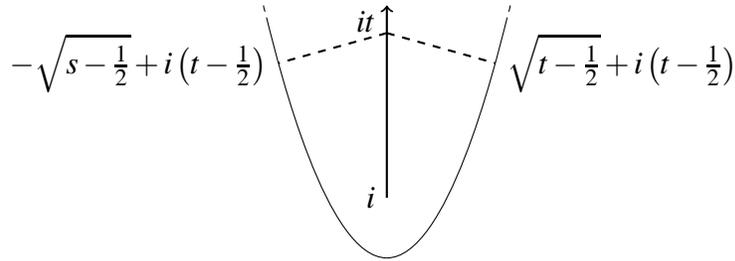	
	Clearly, $(\phi^\alpha_t)$ is  a non-elliptic semigroup in $\mathbb{D}$.
 Since $\bigcup_{t\ge 0}(\Pi_{\alpha}-i t)=\mathbb{C}$, the semigroup is parabolic with zero hyperbolic step. We might assume, without loss of generality, that  $h_\alpha(0)=i$  and $1$ is the Denjoy-Wolff point of $(\phi^\alpha_t)$. The domain $\Pi_\alpha$ is symmetric with respect to the imaginary axis, and therefore by \cite[Thm. 1.1]{BCGDZ} (or \cite[Thm. 17.3.3]{BCDbook}), the orbits of $(\phi_t)$ converge non-tangentially to $1$. Moreover, 
 	\[
	\gamma:[0,+\infty)\longrightarrow\Pi_{\alpha}\quad \text{with} \quad \gamma(t):=i(t+1)
	\] 
is a  geodesic for the hyperbolic distance of $h(\D)$ (see, {\sl e.g.}, \cite[Prop. 6.1.3]{BCDbook}) and $h([0,1))=\gamma([0,+\infty))$ (since $h^{-1}(\gamma(t))\to 1$, as $t\to+\infty$, hence $[0,1)$ and $h^{-1}(\gamma([0,+\infty))$ are geodesics in $\D$, whose closure contain both $0$ and $1$, hence, they are equal). 
	 In particular, the tangential speed of $(\phi_t^\alpha)$ is identically zero, the orthogonal speed $v^o_{\alpha}(t)$  coincides with the total speed $v_{\alpha}(t)$ and, since $\gamma$ is a geodesic,
	\[
	v_{\alpha}(t)=k_{\Pi_{\alpha}}(i,i(1+t))=\int_{1}^{1+t}\kappa_{\Pi_{\alpha}}(is;i)\,\mathrm{d}s.
	\]
	
	By the Distance Lemma for convex simply connected domains (see, {\sl e.g.}, \cite[Thm. 5.2.2]{BCDbook}),
	\begin{equation}\label{valfa}
	\frac{1}{2}\int_{1}^{1+t}\frac{\mathrm{d}s}{\delta_{\alpha}(is)}\le v_{\alpha}(t)\le\int_{1}^{1+t}\frac{\mathrm{d}s}{\delta_{\alpha}(is)},
	\end{equation}
	where $\delta_{\alpha}(ir)$ denotes the Euclidean distance from $ir$ to the boundary of $\Pi_{\alpha}$.
	\begin{lemma}\label{lemma:stimadist}
		Let $\alpha>1$. For  any $c\in (0,1)$, there exists $s_0\geq 1$ such that for all $s\geq s_0$,
		\[
		cs^{1/\alpha}\le\delta_{\alpha}(is)\le s^{1/\alpha}.
		\]
	\end{lemma}
	\begin{proof}
		Fix $s\geq 1$. Since $s^{1/\alpha}$ is the distance of $is$ to the point $s^{1/\alpha}+is\in\partial\Pi_{\alpha}$, it is clear that $\delta_{\alpha}(is)\le s^{1/\alpha}$.
		By the symmetry of $\Pi_{\alpha}$,  there exists $x\geq 0$ such that 
		\[
	\delta_\alpha(is)^2=|(x+ix^\alpha)-is|^2=x^2+(x^{\alpha}-s)^2.
		\] 
In fact, the point $x$ is the largest positive root of the  equation
		\begin{equation}\label{eqmin}
		x^{\alpha}+\frac{1}{\alpha}x^{2-\alpha}-s=0.
		\end{equation}
		Note that, if $1<\alpha\le 2$, this equation has a unique positive root for any $s\ge1$, while, if $\alpha>2$ and $s\ge 1$, there are two positive roots.
				
Now let $x_{\alpha}(s):=x$ be the point defined above. The function $s\mapsto x_{\alpha}(s)$ is strictly increasing and when $s$ goes to infinity, $x_{\alpha}(s)$ diverges to $+\infty$, as well. By \eqref{eqmin},
		\[
		s=x_{\alpha}(s)^{\alpha}\left(1+\frac{1}{\alpha x_{\alpha}(s)^{2(\alpha-1)}}\right)
		\] and one deduces that there exists a positive strictly increasing function $g_{\alpha}(s):[1,+\infty)\rightarrow(0,1)$ such that $\lim_{s\rightarrow+\infty}g_{\alpha}(s)=1$ and
		\[
		\delta_\alpha(is)\geq x_{\alpha}(s)=g_{\alpha}(s)\cdot s^{1/\alpha}.
		\]
Thus the proof is completed. 
	\end{proof}
	\begin{remark}
		If $\alpha=2$, we have \[
		\delta_2(is)=\sqrt{s-\frac{1}{2}+\left(-\frac{1}{2}\right)^2}=\sqrt{s-\frac{1}{4}}.
		\]
	\end{remark}
	Now we can apply Lemma \ref{lemma:stimadist} to \eqref{valfa}. Since \[
	\int_{1}^{1+t}s^{-1/\alpha}\,\mathrm{d}s=\left(\frac{\alpha}{\alpha-1}\right)\left[-1+(1+t)^{1-\frac{1}{\alpha}}\right],
	\]
	for any $\epsilon>0$ and for sufficiently large $t$ (depending on $\epsilon$ and $\alpha$),
	\begin{equation}\label{valfa2}
	\frac{1}{2}\left(\frac{\alpha}{\alpha-1}\right) t^{1-\frac{1}{\alpha}}\lesssim v_{\alpha}(t)= v^o_{\alpha}(t)\lesssim (1+\epsilon)\left(\frac{\alpha}{\alpha-1}\right) t^{1-\frac{1}{\alpha}},
	\end{equation}
	where $f_1(t)\lesssim f_2(t)$ means that there exists $\lambda\in\R$,  such that $f_2(t)-f_1(t)\ge \lambda$ for all $t$.
	
As a direct application of Theorem~\ref{Thm:main2}, \eqref{valfa2} and \eqref{Eq:iper-eucl}, we get the following result.
\begin{proposition}\label{prop:parext}
Suppose $(\phi_t)$ is a non-elliptic semigroup in $\D$ with Denjoy-Wolff point $\tau$ and Koenigs function $h$. Let $v^o(t)$ be the orthogonal speed of $(\phi_t)$. 
\begin{enumerate}
\item Suppose that  $h(\D)\subseteq p+\Pi_{\alpha}$, for some $\alpha>1$ and $p\in\mathbb{C}$. Then
		\[
		\liminf_{t\rightarrow+\infty}\left[v^o(t)-\frac{\alpha }{2(\alpha-1)}t^{1-\frac{1}{\alpha}}\right]>-\infty,
		\]
	or, equivalently, there exists $K>0$ such that for all $t\geq 0$,
\[
|\phi_t(0)-\tau|\leq K \exp\left(-\frac{\alpha}{\alpha-1}t^{1-\frac{1}{\alpha}}\right).
\]
\item Suppose that  $p+\Pi_{\alpha}\subseteq h(\D)$, for some $\alpha>1$ and $p\in\mathbb{C}$. Then for any $\epsilon>0$,
		\[
	\limsup_{t\rightarrow+\infty}\left[v^o(t)-\frac{(1+\epsilon)\alpha }{(\alpha-1)}t^{1-\frac{1}{\alpha}}\right]<+\infty
	\]
	or, equivalently, there exists $K(\epsilon)>0$ such that for all $t\geq 0$,
\[
|\phi_t(0)-\tau|\geq K(\epsilon) \exp\left(-\frac{2(1+\epsilon)\alpha}{\alpha-1}t^{1-\frac{1}{\alpha}}\right).
\]
\end{enumerate}	
	\end{proposition}

\bigskip

\noindent{\bf 2.} Let $\alpha>1$ and $\theta\in \left(0,\pi\right]$. We let
 \[
\Xi(\alpha,\theta):=(-\overline{\Ha}\cap\Pi_{\alpha})\cup W(\theta),
\] where $W(\theta):=\left\lbrace z\in \C\mid\arg (z)\in\left(\frac{\pi}{2}-\theta,\frac{\pi}{2}\right) \right\rbrace$.
\begin{figure}[htb]
	\centering
	\begin{tikzpicture}[scale=0.8]
	\draw (0,0) -- (2.3094,4);
	\draw (2.36714,4.1) -- (2.42487,4.2);
	\draw (0,0) parabola (-2,4);
	\draw (-2.02,4.1) parabola (-2.05,4.2);
	\draw[thick, ->] (0,1) node[left]{$i$}--(0,3.94) node[left]{$it$}--(0,4.2);
	\draw[dashdotted] (0.258819,0.965926)--(1.071797,4) node[right]{$H$}--(1.125387,4.2);
	\draw[thick,dashed] (0,3.74)--(1.61947,2.805) node[right]{$\frac{t}{2}+i\frac{t}{2}\sqrt{3}$};
	\draw[thick,dashed] (0,3.74)--(-1.8,3.24) node[left]{{\tiny $-\sqrt{t-\frac{1}{2}}+i\left(t-\frac{1}{2}\right)$}};
	\end{tikzpicture}
	\caption{The domain $\Xi(2,\frac{\pi}{6})$}
\end{figure}
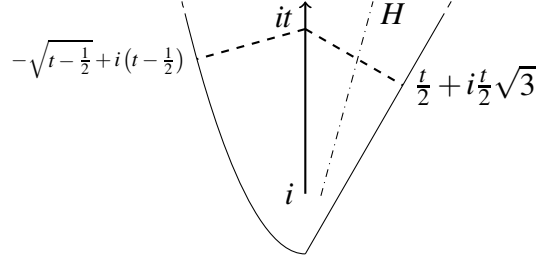

Once again, such a domain is starlike at infinity. It is convex when $0<\theta\leq\frac{\pi}{2}$, otherwise it is  starlike with respect to any  point $z\in \Ha$ with $\arg (z)>0$.  If $h_{\alpha,\theta}:\D\to \Xi(\alpha,\theta)$ is a Riemann map, then it is the Koenigs function of the semigroup $\phi^{\alpha,\theta}_t(z):=h_{\alpha,\theta}^{-1}(h_{\alpha,\theta}(z)+it))$ defined for any $z\in \D$ and $t\geq 0$. As $\bigcup_{t\ge 0}(\Xi(\alpha,\theta)-it)$ is the whole complex plane, the semigroup $(\phi^{\alpha,\theta}_t)$ is parabolic with zero hyperbolic step. Again, we can assume  $h_{\alpha,\theta}(0)=i$, without loss of generality.
For any $t\geq 1$,
\[
\delta_{\alpha,\theta}^+(it):=\min\left\{ \inf\{|z-it|\mid\Re z\ge 0,z\in\C\setminus\Xi(\alpha,\theta)\},t\right\}=\begin{cases}
(\sin\theta)t & \theta\in\left( 0,\frac{\pi}{2}\right) \\
t & \theta\in\left[\frac{\pi}{2},\pi\right]
\end{cases}
\]
while
\[
\delta_{\alpha,\theta}^-(it):=\min\left\{ \inf\{|z-it|\mid\Re z\le 0,z\in\C\setminus\Xi(\alpha,\theta)\},t\right\}=\delta_{\alpha}(it),
\]
where $\delta_{\alpha}(it)$ is the distance from the boundary of $\Pi_{\alpha}$, considered in the first example. By Lemma~\ref{lemma:stimadist}, $\delta_{\alpha}(it)=O(t^{1/\alpha})$, so it follows that the domain is not quasi-symmetric with respect to vertical axes. In particular, by \cite[Thm. 1.1(2)]{BCGDZ},  each orbit of the semigroup $(\phi^{\alpha,\theta}_t)$ converges tangentially to its Denjoy-Wolff point and we can assume that up to conjugation with a rotation, it is equal to $1$.

Let us recall the following result.
\begin{lemma}\cite[Corollary 16.2.6]{BCDbook}\label{lemma:angles}
	Let be $\theta,\eta\in[0,\pi]$, not both equal to zero. Consider the domain \[
	W(\theta,\eta)=\left\lbrace z\in \C\mid\arg (-iz)\in\left(-\theta,\eta\right) \right\rbrace.
	\]
	Let $(\phi_t)$ be a semigroup of holomorphic self-maps in $\D$ with Koenigs map $h$ and $h(\D)=p+W(\theta,\eta)$, for some $p\in\C$.
	\begin{enumerate}
	\item
	If both $\theta$ and $\eta$ are non-zero, the tangential speed $v^T(t)$ of $(\phi_t)$ is bounded, while for the total and orthogonal speeds one has
	\[
	v(t)\sim v^o(t)\sim \frac{1}{2}\left(\frac{\pi}{\theta+\eta}\right)\log t.
	\]
	\item
	If otherwise $\theta\in(0,\pi]$ and $\eta=0$, the speeds of $(\phi_t)$ have the following behavior \[
	v^T(t)\sim \frac{1}{2}\log t,\quad v^o(t)\sim\frac{\pi}{2\theta}\log t,\quad
	v(t)\sim\frac{\pi+\theta}{2\theta}\log t.
	\]
	When $\theta=0$ and $\eta\in (0,\pi]$, the result is analogous, just replace $\theta$ with $\eta$.
	\end{enumerate}
\end{lemma}
Returning to our domain $\Xi(\alpha,\theta)$, we have that $W(\theta)\subset\Xi(\alpha,\theta)$. Moreover, for any $\eta\in (0,\pi]$ we can find a point $p_\eta\in\C$, for which $\Xi(\alpha,\theta)\subset p_\eta + W(\theta,\eta)$. So by Lemma~\ref{lemma:angles} and Theorem~\ref{Thm:main2}, it follows that for the orthogonal speed $v^o_{\alpha,\theta}$ of $(\phi_t^{\alpha,\theta})$ one has 
\begin{equation}\label{eq:orthxi}
\frac{\pi}{2\theta}\left(1-\epsilon\right)\log t\lesssim v^o_{\alpha,\theta}(t)\lesssim \frac{\pi}{2\theta}\log t,
\end{equation}
where $\epsilon:=\frac{\eta}{\theta+\eta}\in\left( 0, \frac{\pi}{\theta+\pi}\right]$ is arbitrarily small, for $\eta$ sufficiently close to zero. More generally, by the same argument, we have an analogous outcome to Proposition \ref{prop:parext}.

\begin{proposition}\label{prop:stimexi}
	Suppose $(\phi_t)$ is a non-elliptic semigroup in $\D$ with Denjoy-Wolff point $\tau$ and Koenigs function $h$. Let $v^o(t)$ be the orthogonal speed of $(\phi_t)$. 
	\begin{enumerate}
		\item\label{1st} Suppose that  $h(\D)\subseteq p+\Xi(\alpha,\theta)$, for some $\alpha>1$, $\theta\in(0,\pi]$ and $p\in\C$. Then for any $\epsilon\in\left( 0, \frac{\pi}{\theta+\pi}\right] $
		\[
		\liminf_{t\rightarrow+\infty}\left[v^o(t)-\frac{\pi}{2\theta}(1-\epsilon)\log t\right]>-\infty,
		\]
		or, equivalently, there exists $K(\epsilon)>0$ such that for all $t\geq 0$,
		\[
		|\phi_t(0)-\tau|\leq K(\epsilon) t^{(-1+\epsilon)\pi/\theta}.
		\]
		\item\label{2nd} Suppose that  $p+\Xi(\alpha,\theta)\subseteq h(\D)$, for some $\alpha>1$, $\theta\in(0,\pi]$ and $p\in\C$. Let's assume that $h(\D)$ is starlike with respect to an inner point. Then
		\[
		\limsup_{t\rightarrow+\infty}\left[v^o(t)-\frac{\pi}{2\theta}\log t\right]<+\infty
		\]
		or, equivalently, there exists $K>0$ such that for all $t\geq 0$,
		\[
		|\phi_t(0)-\tau|\geq K t^{-\pi/\theta}.
		\]
	\end{enumerate}	
\end{proposition}
\begin{remark}
	The results above  do not depend on $\alpha$. This is not a deficiency of the methods we use, but a natural fact, due to \eqref{eq:orthxi}. In other words, in the previous setting, the ``non-tangential'' side controls the orthogonal speed. Indeed, Condition
 \eqref{2nd} of Proposition \ref{prop:stimexi} is equivalent to assume  the (weaker) hypothesis that $p+W(\theta)\subseteq h(\D)$ and $h(\D)$ is starlike.
\end{remark}

On the other hand, it is interesting to note that the exponent $\alpha$ controls the tangential speed of the semigroup $(\phi^{\alpha,\theta}_t)$, which is not influenced by the angle $\theta$.
\begin{proposition}
For the tangential speed of the semigroup $(\phi^{\alpha,\theta}_t)$, the following estimates (up to real constants) hold
\[
\frac{1}{4}\left(1-\frac{1}{\alpha}\right)\log t \lesssim v^T_{\alpha,\theta}(t)\lesssim \frac{1}{2}\left(1-\frac{1}{\alpha}\right)\log t.
\]
Hence for any $\epsilon \in \left( 0, \frac{\pi}{\theta+\pi}\right]$, we have the following bounds for the total speed
\[
\left(\frac{\pi}{2\theta}(1-\epsilon)+\frac{1}{4}-\frac{1}{4\alpha}\right)\log t \lesssim v_{\alpha,\theta}(t)\lesssim \left(\frac{\pi}{2\theta}+\frac{1}{2}-\frac{1}{2\alpha}\right)\log t.
\]
\end{proposition}

\begin{proof} We divide the proof into steps.
\medskip

\emph{Step 1. Lower bound for tangential speed}.

Let $H$ be the curve \[ H:[1,\infty)\longrightarrow \Xi(\alpha,\theta)\quad \text{with} \quad H(r)= re^{i(\pi-\theta)/2}.\]
This curve is a quasi-geodesic, as its hyperbolic length is 
\[
\ell_{\Xi(\alpha,\theta)}(H,[r_1,r_2])\leq \ell_{W(\theta)}(H,[r_1,r_2])\leq \int_{r_1}^{r_2}\frac{\mathrm{d}r}{\delta_{W(\theta)}(H(r))}=\frac{1}{\sin\frac{\theta}{2}}\log\frac{r_2}{r_1}
\] and by the Distance Lemma for simply connected domains (see, \emph{e.g.} \cite[Thm. 3.5]{BCGDZ})
\[
k_{\Xi(\alpha,\theta)}(H(r_1),H(r_2))\geq \frac{1}{4}\log \left(1+\frac{r_2-r_1}{\left(\sin\frac{\theta}{2}\right)r_1}\right)\geq \frac{1}{4}\log\frac{r_2}{r_1}.
\]
By means of the Gromov shadowing Lemma (see, {\sl e.g.}, \cite[Thm. 6.3.8]{BCDbook}), it is enough to find bounds for \[
\inf_{r\ge 1}k_{\Xi(\alpha,\theta)}(it,H(r)),
\] since the same bounds, up to constants not depending on $t$, will hold also for $v^T_{\alpha,\theta}(t)$.

Now, once $t$ is chosen big enough, $\delta_{\Xi(\alpha,\theta)}(it)=\delta_{\alpha}(it)=O(t^{1/\alpha})$. If $\left(\sin\frac{\theta}{2}\right)r \leq \delta_\alpha(it)$, then 
\begin{equation*}
\begin{split}
k_{\Xi(\alpha,\theta)}(it,H(r))&\geq \frac{1}{4}\log \left(1+\frac{|it-H(r)|}{\left(\sin\frac{\theta}{2}\right)r}\right)\\&=\frac{1}{4}\log \left(1+\frac{\sqrt{r^2+t^2-2rt\cos\frac{\theta}{2}}}{\left(\sin\frac{\theta}{2}\right)r}\right)\geq\frac{1}{4}\log \left(1+\frac{t-r}{\left(\sin\frac{\theta}{2}\right)r}\right),
\end{split}
\end{equation*}
which is a decreasing function of $r$, so 
\[
\inf_{1\le r\le \left(\sin\frac{\theta}{2} \right)^{-1}\delta_\alpha(it) }k_{\Xi(\alpha,\theta)}(it,H(r))\ge \frac{1}{4}\log \left(1+\frac{t-\left(\sin\frac{\theta}{2} \right)^{-1}\delta_\alpha(it)}{\delta_\alpha(it)}\right)\sim \frac{1}{4}\left(1-\frac{1}{\alpha}\right)\log t.
\]
On the other hand, if $\left(\sin\frac{\theta}{2}\right)r>\delta_\alpha(it)$, then 
\[
k_{\Xi(\alpha,\theta)}(it,H(r))\geq \frac{1}{4}\log \left(1+\frac{|it-H(r)|}{\delta_\alpha(it)}\right)\geq \frac{1}{4}\log \left(1+\frac{\left(\sin\frac{\theta}{2}\right)t}{\delta_\alpha(it)}\right)
\]
and so, one concludes that $v^T_{\alpha,\theta}(t)\gtrsim \frac{1}{4}(1-1/\alpha)\log t$.

\medskip
\emph{Step 2. Upper bound for tangential speed}.

For every $t$ greater than some fixed $t_0\geq 1$, $\delta_{\Xi(\alpha,\theta)}(it)=\delta_\alpha(it)$ and the point $q_t:=it+\delta_\alpha(it)$ belongs to $W(\theta)\subset \Xi(\alpha,\theta)$. So for any $t\geq t_0$ and $r\geq 1$, we define the path $\sigma_{t,r}$ given by the concatenation of the Euclidean segment from $it$ to $q_t$
\[
L_t:[0,1]\longrightarrow \Xi(\alpha,\theta)\quad \text{with} \quad L_t(s)=it+\delta_\alpha(it)s,
\]
where $\gamma_{t,r}$  is the geodesic arc with respect to the hyperbolic metric of $W(\theta)$ joining $q_t$ with $H(r)=re^{i(\pi-\theta)/2}$. By possibly increasing $t_0$, we may also assume that $\delta_{\Xi(\alpha,\theta)}(L_t(s))\geq \delta_\alpha(it)$, for any $0\leq s\leq 1$. Therefore we have 
\begin{equation*}
\begin{split}
k_{\Xi(\alpha,\theta)}(it,H(r))&\leq \ell_{\Xi(\alpha,\theta)}(\sigma_{t,r})=\ell_{\Xi(\alpha,\theta)}(L_t)+\ell_{\Xi(\alpha,\theta)}(\gamma_{t,r})\leq\ell_{\Xi(\alpha,\theta)}(L_t)+\ell_{W(\theta)}(\gamma_{t,r})\\&
\leq
\int_0^1\frac{\delta_\alpha(it)\,\mathrm{d}s}{\delta_{\Xi(\alpha,\theta)}(L_t(s))}+k_{W(\theta)}(q_t,H(r))\leq\int_0^1\frac{\delta_\alpha(it)\,\mathrm{d}s}{\delta_\alpha(it)}+k_{W(\theta)}(q_t,H(r))\\&=1+k_{W(\theta)}(q_t,H(r)).
\end{split}
\end{equation*}
Now let $\beta_t:=\frac{\pi}{2}-\arg q_t$, so that $t\cdot\tan\beta_t=\delta_\alpha(it)$. Thus, considering the conformal map $z\mapsto z^{\pi/\theta}$ which sends $\widetilde{W}(\theta):=e^{i(\theta-\pi)/2}W(\theta)$ onto $\Ha$, and using known estimates for $k_{\Ha}$ (see for instance \cite[Lemma 2.1]{Br})
\begin{equation*}
\begin{split}
k_{W(\theta)}(q_t,H(r))&=k_{W(\theta)}(|q_t|e^{i\left(\frac{\pi}{2}-\beta_t\right)},re^{i(\pi-\theta)/2})=k_{\widetilde{W}(\theta)}(|q_t|e^{i\left(\frac{\theta}{2}-\beta_t\right)},r)\\&=
k_{\widetilde{W}(\theta)}\left(1,\frac{|q_t|}{r}e^{i\left(\frac{\theta}{2}-\beta_t\right)}\right)=k_{\Ha}\left(1,\frac{|q_t|^{\pi/\theta}}{r^{\pi/\theta}}e^{i\left(\frac{\pi}{2}-\frac{\pi}{\theta}\beta_t\right)}\right)\\&\le \frac{\pi}{2\theta}\log\frac{|q_t|}{r}+\frac{1}{2}\log\frac{1}{\sin\left(\frac{\pi}{\theta}\beta_t\right)}+\frac{1}{2}\log 2.
\end{split}
\end{equation*}
By choosing $r=|q_t|=t\sqrt{1+\tan^2\beta_t}$ and by observing that, since $\lim_{t\rightarrow+\infty}\beta_t=0$ \[
\log\frac{1}{\sin\left(\frac{\pi}{\theta}\beta_t\right)}\sim\log\frac{\theta}{\pi\beta_t}\sim\log\frac{1}{\tan\beta_t}=\log\frac{t}{\delta_\alpha(it)}\sim\left(1-\frac{1}{\alpha}\right)\log t,
\]
we conclude that
\[
v^T_{\alpha,\theta}(t)\leq k_{\Xi(\alpha,\theta)}(it,H(|q_t|))\lesssim\frac{1}{2}\left(1-\frac{1}{\alpha}\right)\log t.
\]

\bigskip
\emph{Step 3. Total speed}.

The statement for the total speed $v_{\alpha,\theta}$  follows directly from \eqref{Pyt} and \eqref{eq:orthxi}.
\end{proof}

\section{Final remarks and open questions}\label{sec:open}

Theorem~\ref{Thm:almost-asymp-mono} and Theorem~\ref{Thm:almost-asymp-mono1} move towards the direction of giving an affirmative answer to {\sl Question~4} in \cite{Br}. However, the complete answer is still unknown, and, as it follows from the results in Section~\ref{Sec:estimates}, if counterexamples exist, they are rather peculiar. 

Note that (using the same notation as in Section~\ref{Sec:estimates}), given a semigroup $(\phi_t)$ of $\D$, since $\rho_s\to +\infty$, as $s\to +\infty$, by the same argument of Lemma~\ref{lem:good-estim-increase}, for all $t\geq 1$, there exists $s_t\geq t$ such that
\[
\inf_{s\geq t}\omega(\rho_s e^{i\theta_s}, \Theta_{t}, \Ha)=\omega(\rho_{s_t} e^{i\theta_{s_t}}, \Theta_{t}, \Ha)
\]
and 
\[
\liminf_{s\to +\infty}\omega(\rho_s e^{i\theta_s}, \Theta_{t}, \Ha)>0.
\]

\medskip
{\sl Question (i)}: Does there exist a semigroup $(\phi_t)$ of $\D$ so that \[
\liminf_{t\to+\infty} \omega(\rho_{s_t} e^{i\theta_{s_t}}, \Theta_{t}, \Ha)=0?
\]
\medskip

By the  results in Section~\ref{Sec:estimates}, if such a semigroup exists, the orbits do not converge non-tangentially to the Denjoy-Wolff point. Then the orthogonal---and hence the total---speed is not (eventually) non-decreasing. This raises the second question:

\medskip
{\sl Question (ii)}: Does there exist a semigroup $(\phi_t)$ of $\D$ so that the orthogonal speed is not  (eventually) non-decreasing? Note that this is equivalent to ask if $t\mapsto \rho_t$ is not (eventually) non-decreasing, for a semigroup $(\phi_t)$ of $\D$.

\end{document}